\documentclass{amsart}
\usepackage{amsmath,amssymb,latexsym,color}
\newtheorem{cor}{Corollary}[section]

     \newtheorem{lemma}[cor]{Lemma}

     \newtheorem{prop}[cor]{Proposition}
     
     \newtheorem{thm}[cor]{Theorem}
     \newtheorem{defi}[cor]{Definition}

     \newcommand{\Alt}[1]{\mathrm{Alt}(#1)}

     \newcommand{\Aut}[1]{\mathrm{Aut}(#1)}
     \newcommand{\ov}[1]{\overline{#1}}
     \newcommand{\nor}{\unlhd}

     \newcommand{\PSL}{\mathrm{L}}                 
     \newcommand{\PSU}{\mathrm{U}}
     \newcommand{\PSp}{\mathrm{S}}
     \newcommand{\OO}{\mathrm{O}}

     \newcommand{\Sym}[1]{\mathrm{Sym}(#1)}
     
     \newcommand{\wrr}{\ \mathrm{wr}\ }

\begin{document}
    \author[Fumagalli]{Francesco Fumagalli}
    \address{Dipartimento di Matematica  e Informatica ``U.\,Dini", \\
   Universit\`a di Firenze\\
   Viale Morgagni 67A, \ I-50134 Firenze, Italy}
   \email{francesco.fumagalli\,@\,unifi.it}
   
   \author[Leinen]{Felix Leinen}
   \address{Institute of Mathematics, \\
   Johannes Gutenberg-University,\\
   D$-$55099 \,Mainz, Germany}
   \email{Leinen\,@\,uni-mainz.de}
   
   \author[Puglisi]{Orazio Puglisi}
   \address{Dipartimento di Matematica  e Informatica ``U.\,Dini", \\
   Universit\`a di Firenze\\ 
   Viale Morgagni 67A, \ I-50134 Firenze, Italy}
   \email{orazio.puglisi\,@\,unifi.it}

\title{A reduction Theorem for nonsolvable finite groups}
\begin{abstract}
Every finite group $G$ has a normal series each of whose factors is either a solvable
group or a direct product of nonabelian simple groups. The minimum number of
nonsolvable factors attained on all possible such series is called the nonsolvable
length of the group and denoted by $\lambda(G)$. 
For every integer $n$, we define a particular class of 
groups of nonsolvable length $n$, called \emph{$n$-rarefied}, 
and we show that every finite group of nonsolvable length $n$ contains 
an $n$-rarefied subgroup. As applications of this result, we improve the 
known upper bounds on $\lambda(G)$ and determine the maximum possible 
nonsolvable length for permutation groups and linear groups of fixed degree
resp. dimension. 
\end{abstract}
\maketitle

\section{Introduction}
Every finite group $G$ has a normal series each of whose factors is 
either a solvable group or a direct product of nonabelian simple groups. 
The minimum number of nonsolvable factors attained on all possible 
such series is called  \emph{the nonsolvable length of $G$} and denoted 
by $\lambda(G)$ (see \cite{khukhro}). 
In a series of recent papers (
see also \cite{khukhro2}, \cite{khukhro3}, 
\cite{DS}, \cite{CS1}, \cite{CS2}) this parameter has been investigated, 
as some bounds on $\lambda(G)$ have been proved to be useful in several situations.

The aim of this paper is it to provide a tool that might be useful when
investigating  properties related to the nonsolvable length. 
For every natural number $n$ denote $\Lambda_n$ the class of groups of 
nonsolvable length $n$.  We define a particular kind 
of $\Lambda_n$-groups, called \emph{$n$-rarefied} (see Definition \ref{minimal}) 
whose structure of normal subgroups is quite restricted. 
E.g. any chief series of an $n$-rarefied group has exactely $n$ nonabelian 
factors and the nonabelian composition factors are simple groups 
isomorphic to one of the following:
$$\PSL_2(2^r),\, \PSL_2(3^r),\, \PSL_2(p^a),\, \PSL_3(3),\, \null^2B_2(2^s), $$
where $p,r,s$ are primes, $p,s$ odd and $a\geq 0$.

The main result of this paper is
\begin{thm}\label{main} 
Every finite group in $\Lambda_n$ contains an $n$-rarefied subgroup.
\end{thm}

\noindent
The proof relies on the Classification of Finite Simple Groups.

\noindent
Due to the peculiar structure of $n$-rarefied groups, Theorem \ref{main} plays a 
central r\^{o}le whenever we deal with questions related to the nonsolvable length 
that can be reduced to subgroups. In the last section of the paper we 
provide some examples of this situation. In particular, we use the above ideas to 
improve the bound on $\lambda(G)$, obtained in \cite[Theorem 1.1.(a)]{khukhro}\\

\noindent 
{\bf Theorem \ref{bound1} } {\sl Let $G$ be any finite group. Then $\lambda(G)\leq 
L_2(G).$}\\

Here $L_2(G)$ denotes the maximum of the $2$-lengths of all possible
solvable subgroups of $G$. Also, along the lines of \cite{LMS}, 
where the authors study group properties that can be detected from the 
$2$-generated subgroups, we prove \\

\noindent 
{\bf Theorem \ref{generators} }{\sl 
Let $G$ be any finite group. If $G$ is not solvable then there exists a 
$2$-generator subgroup $H$ of $G$ such that $\lambda(H)=\lambda(G)$.}\\

The last application of Theorem \ref{main} concerns the nonsolvable length of 
subgroups of $H=\mathrm{Sym}(m)$ resp. of $H=\mathrm{GL}(m, \mathbb F)$, for any fixed  
natural number $m$ and any field $\mathbb F$. 
Denoting by $\lambda(m)$ resp. $\lambda_{\mathbb{F}}(m)$ 
the maximum value of $\lambda(G)$ when $G$ is a subgroup of $H$,
we prove the following\\

\noindent 
{\bf Theorem \ref{lambda(m)} }{\sl 
For every $m\geq 5$ we have $\lambda(m)=\lfloor \log_5(m)\rfloor$.}\\

\noindent 
{\bf Theorem \ref{lambda_F(m)} }{\sl  
For every $m\geq 2$ and every field $\mathbb{F}$ with at least four elements, 
we have $\lambda_{\mathbb{F}} (m)=1+\lfloor \log_5(m/2)\rfloor$. 
When $\vert \mathbb{F}\vert \leq 3$, then 
$1+\lfloor \log_5(m/3)\rfloor \leq \lambda_{\mathbb{F}} (m)\leq 1+\lfloor \log_5(m/2)\rfloor$. }

\section{Preliminaries}

In this section we set up the basic definitions and notation, and collect a series 
of technical facts to  be used in the proof of the main result (Theorem 
\ref{main}). 

\begin{defi}
Let $G$ be any finite group. Define
$$
R(G)=\langle B\mid B \textrm{ is a  normal solvable  subgroup of } G\rangle.
$$

\end{defi}

The subgroup $R(G)$ is the \emph{solvable radical} of $G$. It is clearly normal and 
solvable, and $G/R(G)$ does not have any nontrivial normal solvable subgroup.


\begin{defi} Let $G$ be any finite group. Define
$$
S(G)=\langle A\mid A \textrm{ is a minimal normal nonabelian subgroup of } G\rangle.
$$
\end{defi}

If  $S(G)\not=1$, then it  is the direct product of nonabelian simple groups. 
All its  factors are subnormal in $G$ and they are called the \emph{components} 
of $G$. It is easy to check that $S(G)$ is the subgroup generated by all the non 
abelian simple subnormal subgroups of $G$.

A series can then be defined as follows.

\begin{defi}Let $G$ be any finite group. Set
$R_1(G)=R(G)$ and define $S_1(G)$ by the equation
$$ S_1(G)/R_1(G)=S(G/R_1(G)). $$
Once $R_i(G)$ and $S_i(G)$ have been defined for all $i<n$, the subgroups $R_n(G)$ 
and $S_n(G)$ are determined by the equations

$$ R_n(G)/S_{n-1}(G)= R(G/S_{n-1}(G)) $$
and
$$ S_n(G)/R_n(G)= S(G/R_n(G)). $$
We call this series the \emph{$RS$}-series of $G$.
\end{defi}

The series thus defined always reaches $G$. 
If $G=R_{n+1}(G)>R_n(G)$, we say that $G$ has 
\emph{nonsolvable length} $\lambda(G)=n$ (according to the notation introduced in 
\cite{khukhro}). Also, for $n\geq 1$ we denote by $\Lambda_n$ the class 
of finite groups $G$ of nonsolvable length $n$. 


Note that the class of groups having nonsolvable length zero coincides with 
the one of solvable groups. Also, simple/quasisimple/almost simple groups, 
as well as their direct products are all groups of nonsolvable length one. 
A tipical example of a group of nonsolvable length $n$ is an $n$-fold wreath 
product of a fixed nonabelian simple group.

The next Lemma collects some easy but useful observations. 
The proof is left to the reader.
\begin{lemma}\label{trivial}
Let $G$ be a finite group and assume that $\lambda(G)=n$. 
The following hold.
\begin{enumerate}
\item  $\lambda(G/R_i(G))=n-i+1$, for all $i=1, \dots, n$.
\item  $\lambda(G/S_i(G))=n-i$, for all $i=1, \dots, n$.
\item  If $N\trianglelefteq G$, then $\lambda(G/N)\leq 
\lambda(G)\leq 
\lambda(N)+\lambda(G/N)$.
\item  $C_{G/R_i(G)}(S_i(G)/R_i(G))=1$, for all $i=1, \dots , n$.
\item  If $H$ is a subnormal subgroup of $G$ then $R_i(H)\leq R_i(G)$ 
       and $S_i(H)\leq S_i(G)$, for all $i=1,\ldots,n$.
\item The $RS$-series of $G$ is a shortest series whose 
factors are either solvable or semisimple.
\end{enumerate}
\end{lemma}

In the course of our analysis of $\Lambda_n$-groups 
acting on sets/vector spaces we will  need the following result. 
\begin{lemma}\label{subdirect3}  Let $G$ be a finite group with  normal 
subgroups $N_1, N_2,\ldots, N_r$ such that $\bigcap_{i=1}^rN_i=1$. 
Then $\lambda(G)=\max\{\lambda(G/N_i)\vert i=1, \dots, r\}$.
\end{lemma}
\begin{proof} For every $i=1,\ldots, r$ call $G_i=G/N_i$ and set also $n$ the greatest of 
all $\lambda(G_i)$. 
Since $G$ has an homorphic image of nonsolvable length $n$, we have 
by Lemma \ref{trivial}(3), $\lambda(G)\geq n$.
Note that $\lambda(G)\leq n$ is of course true if $\lambda(G)=0$ or $1$, since 
if all the $G_i$ are solvable groups then $G$ itself is solvable. 
Assume that $G$ is a minimal counterexample and therefore $\lambda(G)>1$.   
For each $i$ let $A_i$ be such that $A_i/N_i$ is the solvable radical of $G_i$, 
and let $A=\cap_{i=1}^rA_i$. Then, since $A$ embeds into the direct product 
$\prod_{i=1}^r A_i/N_i$, $A$ is solvable and, being normal, $A\leq R_1(G)$. 
For every $i$, $R_1(G)N_i/N_i$ is a solvable normal subgroup of $G_i$, therefore 
$R_1(G)\leq A_i$ for every $i$, hence $R_1(G)=A$. 
Now, if $R_1(G)\not=1$, the claim is proved by induction because 
$\lambda(G)=\lambda(G/R_1(G))$ and $\lambda(G_i)=\lambda(G/A_i)$ 
for all $i=1,\ldots,r$. We assume therefore that $R_1(G)=1$ and moreover, 
being  $G$ a subdirect product of the groups $G/A_i$, we can assume 
that each $G/N_i$ has trivial solvable radical, i.e. $N_i=A_i$ for every $i$. 
For each $i$ let $K_i$ be the subgroup of $G$ defined by $K_i/N_i=S_1(G_i)$ and 
let $K=\bigcap_{i=1}^r K_i$. Of course, as $S_1(G)N_i\leq K_i$ for each $i$, we have 
that $S_1(G)\leq K$. In particular, $G/K$ has order strictly smaller than 
the order of $G$ and it is a subdirect product of the groups 
$G/K_i\simeq G_i/S_1(G_i)$, for $i=1,\ldots,r$. Hence by the inductive assumption 
we have that $\lambda(G/K)=\max\{\lambda(G_i/K_i)\vert i=1,\ldots, r\}=n-1$.  
Also $K$ is a subdirect product of the groups $K_i/N_i$, for $i=1,\ldots,r$. Hence 
if $K\neq G$, by the inductive assumption we have that 
$\lambda(K)=\max\{\lambda(K_i/N_i)\vert i=1,\ldots, r\}=1$. 
By Lemma \ref{trivial} {\it (3)} it follows that $\lambda(G)\leq n$. 
Therefore we have that $G=K=K_i$ for every $i=1,\ldots,r$. 
In particular, $G_i=S_1(G_i)$ for every $i$ and so $G$ is a subdirect 
product of semisimple groups. 
By \cite[Lemma 4.3A]{dixon} $G$ is a direct product of simple groups, then 
$\lambda(G)=1=\max\{\lambda(G_i)\vert i=1,\ldots,r\}$ which completes the proof.
\end{proof} 

\section{Some technical results on finite simple groups}

Given two finite groups $X$ and $Y$ with $X\nor Y$, let $K$ be a subgroup of $X$. 
We say that $K$ {\it extends from $X$ to $Y$} if $Y= XN_Y(K)$. This is equivalent 
to say that the $Y$-conjugacy class of $K$, which we denote by 
$[K]_Y=\{K^y\vert y\in Y\}$,  coincides with the $X$-conjugacy class, $[K]_X$.
If, for example, $X$ contains a unique conjugacy class of subgroups isomorphic
to $K$, then $K$ extends from $X$ to $Y$ by a Frattini argument.\\ 
The proof of the following Lemma is left to the reader.

\begin{lemma}\label{ext_basic}
Assume that $K, X$ and $Y_1$ are subgroups of a group $Y$ such that $K < X < Y_1 < Y$.
\begin{enumerate}
\item If $K$ extends from $X$ to $Y$, 
then $K$ extends from $X$ to $Y_1$.
\item If $K$ is maximal in $X$ and $K$ extends from $X$ to $Y$, then $N_Y(K)$ is
maximal in $Y$ provided $K$ is not normal in $Y$.
\end{enumerate}
\end{lemma}

Beside the standard notation generally used in literature, we shall also 
make use of the one estabilished in the Atlas (\cite{atlas}), e.g.:
\begin{itemize}
\item $n$ denotes both a natural number and the cyclic group of that order;
\item $A.B$ denotes an extension of the group $A$ by the group $B$;
\item $A:B$ denotes a split extension;
\item $A^{\cdot} B$ denotes a non-split extension.
\end{itemize} 
The notation for sporadic simple groups and simple groups of Lie type 
follows the Atlas, while $\Alt{m}$ and $\Sym{m}$ denote respectively 
the alternating and symmetric group of degree $m$. 
We refer to \cite{kleidman_liebeck} for the definition of the Aschbacher's classes 
${\mathcal{C}}_i$ (for $i=1,\ldots,8$) of natural subgroups of classical groups.\\

From now on $\mathcal{L}$ will denote the class consisting of the following simple groups:
\begin{itemize}
\item $\PSL_2(q)$ for $q\in\{2^r,\, 3^r,\, p^{2^a}\}$ with $r$ a prime, $p$ an odd 
prime and $a\geq 0$,
\item $ \PSL_3(3)$,
\item $\null^2B_2(2^s)$ for $s$ an odd prime.
\end{itemize} 
Note that $\mathcal{L}$ contains $\Alt{5} \simeq \PSL_2(4)$ and 
$\Alt{6}\simeq \PSL_2(9)$. Moreover, every minimal simple group 
(i.e. every simple group whose proper subgroups are all solvable) lies in 
$\mathcal{L}$ (see \cite[Corollary 1]{thompson}).

In the proof of Lemma  \ref{normalizer} we need a technical result 
about the existence of certain subgroups in finite nonabelian 
simple groups not belonging to $\mathcal{L}$.
The proof of this result relies on the classification of finite simple groups.

\begin{lemma}\label{simple1}
Let $S$ be a nonabelian finite simple group. 
Assume that $S$ does not belong to $\mathcal{L}$.
Then $S$ contains a proper
self-normalizing subgroup $H$ which belongs to $\Lambda_1$ and 
such that it extends to $\mathrm{Aut}(S)$.
\end{lemma}
\begin{proof}$\null$

\noindent
{\it Case $S$ alternating.} 

Let $S=\Alt{n}$ with $n\geq 7$.
We may choose $H$ to be the stabilizer of a point. 
Then $H$ is a maximal subgroup of $S$ isomorphic to $\Alt{n-1}$, 
hence $N_S(H)=H\in \Lambda_1$. 
Trivially $H$ exteds to $\Aut{S}=\Sym{n}$.\\

\noindent
{\it Case $S$ a simple group of Lie type in characteristic $p$.} 

Let $G$ be the extension of $S$ by the diagonal and field
automorphisms. We consider separately the two cases: 
\begin{enumerate}
\item $G=\Aut{S}$,
\item $G<\Aut{S}$.
\end{enumerate} 
Case 1. 
By general $BN$-pair theory 
\cite[Proposition 8.2.1 and Theorem 13.5.4]{carter}, 
the lattice of proper overgroups in $S$ of a Borel subgroup $B=N_S(U)$ 
(where $U$ is a fixed Sylow $p$-subgroup of $S$) consists of the 
parabolic subgroups of $S$. In particular the maximal parabolic subgroup,
$P_{\hat{1}}$, that is correlated to the first node of the 
Dynkin diagram (or with the orbit of nodes labelled $1$ in the twisted case), 
is a maximal subgroup of $S$.  
Now the group $G$ admits a $BN$-pair too, whose 
Borel subgroup is $N_{G}(U)$, since in order to construct $G$ 
from $S$ we can choose diagonal and field automorphisms that 
normalize every root subgroup of $U$. 
%
%
In particular, $P_{\hat{1}}$ extends to $G$. 
As long as $P_{\hat{1}}\in \Lambda_1$, we may therefore choose $H=P_{\hat{1}}$ 
to prove the Lemma. 
Note that $P_{\hat{1}}=U_{\hat{1}}L_{\hat{1}}$ is the product of a $p$-subgroup, 
its unipotent radical $U_{\hat{1}}$, by the Levi complement $L_{\hat{1}}$, which 
is a central product of groups of Lie type corresponding to the subdiagram obtained by 
deleting the first node (the first orbit of nodes). Thus 
$P_{\hat{1}}\in \Lambda_1$ if and only if $L_{\hat{1}}$ does and this almost 
always happens. The only exceptions are when the Dynkin diagram is a 
point, or an orbit of two points (in the twisted case), or a pair of points 
and $q\in\{2,3\}$, because this are the cases in which the group $\PSL_2(q)$ 
is solvable. 
Namely, $H$ cannot be chosen to be $P_{\hat{1}}$ in the following situations:
\begin{enumerate}
\item[i.] $\PSL_2(q)$ for all $q\geq 4$,
\item[ii.] $\PSU_3(q)$ for all $q\geq 3$,
\item[iii.] $\PSU_4(3),\, \PSU_5(2), \, \PSU_5(3),$ 
\item[iv.] $^2B_2(2^r)$ with $r$ odd and $2^r\geq 8$,
\item[v.] $^2G_2(3^r)$ with $r$ odd and $3^r\geq 27$,
\item[vi.] $\PSp_4(3),\,  ^3D_4(2),\, ^3D_4(3)$,
\end{enumerate} 
since $\PSp_4(2)\simeq \Sym{6}$, $\PSU_3(2)\simeq 3^2:Q_8$, 
$\PSU_4(2)\simeq \PSp_4(3)$, $G_2(2)\simeq \PSU_3(3).2$, 
$^2B_2(2)\simeq 5:4$, $^2G_2(3) \simeq \PSL_2(8).3$ and 
$^2F_4(2)$ has a simple normal subgroup of index 2.

i. Let $S=\PSL_2(q)$ for $q=p^f\geq 4$. We may now consider maximal subgroups of $S$ 
belonging to Aschbacher's class $\mathcal{C}_5$. In particular, if $p=2$ and $f$ is not 
a prime there exists a maximal subgroup $H$ isomorphic to $\PSL_2(q_0)$, 
with $q=q_0^r$ for some prime $r$ dividing $f$, that extends to $\Aut{S}$ 
(see for instance \cite[Table 8.1]{colva}). Similarly, if $p$ is odd and there 
exists an odd prime $r$ dividing $f$, then there is a maximal subgroup $H$ 
isomorphic to $\PSL_2(q_0).2$, with $q=q_0^r$, that extends to $\Aut{S}$. 
Note that such a subgroup lies in $\Lambda_1$, provided that $q_0\neq 3$. 
Thus being
$$\{\PSL_2(2^r),\, \PSL_2(3^r),\, 
\PSL_2(p^{2^a})\vert r,p \textrm{ primes},\, p \textrm{ odd},\, a\geq 0  \}
\subseteq \mathcal{L}$$
the Lemma is proved in this case.

We treat together ii. and iii. (and in general every unitary case). \\
Let $S=\PSU_n(q)$, with $n\geq 3$. Let $H$ be the stabilizer of a non-isotropic 
point (in the natural unitary $n$-dimensional space). Then  $H$ is a maximal 
subgroup of $S$ and it extends to $\Aut{S}$ 
(see \cite[Table 3.5.B]{kleidman_liebeck} and \cite[Table 8]{colva}). 
Also, $H'$ is a cyclic extension of $\PSU_{n-1}(q)$, therefore $H\in\Lambda_1$,
whenever $(n,q)\ne (3,3)$. For $S=\PSU_3(3)$, a direct inspection in 
\cite{atlas} suggests that we may take $H$ to be the maximal subgroup $\PSL_2(7)$.

iv. Let $S= \null^2B_2(2^r)$ with $r\geq 3$ odd and composite. Then if $l$ is a prime divisor
of $r$ the centralizer in $S$ of a field automorphism 
of order $l$ is a maximal 
subgroup of $S$, isomorphic to $^2B_2(2^{r/l})$ that extends to $\Aut{S}$ 
(see \cite{suzuki}).

v. Let $S=\null^2G_2(3^r)$ with $r$ odd and $r\geq 3$. Then $S$ has a unique 
class of involutions (see \cite{kleidman_G_2}). Let $i$ be one involution,
then $C_S(i)$ is a maximal subgroup of $S$ isomorphic to $2\times \PSL_2(q)$  
(\cite[Theorem C]{kleidman_G_2}) and therefore it can be chosen as $H$.

Finally consider the three cases in vi.

Let $S=\PSp_4(3)$. In this case we may take $H\simeq 2^4:\Alt{5}$ (see \cite{atlas}). 

Let $S=\, ^3D_4(2)$. A direct inspection in \cite{atlas} suggests that we may take 
$H\simeq 2^9:\PSL_2(8)$.

Let $S=\, ^3D_4(3)$. By \cite{kleidman_3D_4} we may take as $H$
a copy of the maximal subgroup $3^9:\mathrm{SL}_2(27).2$.\\

Case 2. Suppose now that $G<\Aut{S}$, that is, there exist graph 
automorphisms of $S$. Note that this happens exactly 
when $S$ is one of the following groups (see \cite{carter} or \cite{atlas}):

$$\PSL_n(q) \textrm{ for } n\geq 3,\,\, 
\PSp_4(2^r),\,\, 
\OO_{2n}^+(q) \textrm{ for } n\geq 4,\,\, 
G_2(3^r),\,\, 
F_4(2^r),\,\, 
E_6(q).$$

Let $S=\PSL_n(q)$ with $n\geq 3$. Let $\iota$ be a graph automorphism 
acting as the inverse-transpose on the elements of $S$, 
and let $H=P_{\hat{1}}\cap P_{\hat{1}}^{\iota}$.
Note that $H$ is the stabilizer in $S$ of a direct decomposition of the 
natural module $V=U\oplus W$, where $U$ is  $1$-dimensional and 
$W$ $(n-1)$-dimensional. 
In particular we have that $H=N_S(H)$, since $n-1\geq 2$. 
Moreover, $H$ extends to $\Aut{S}$ and $H'$ is a central extension 
of $\PSL_{n-1}(q)$ (see \cite[Prop. 4.1.4]{kleidman_liebeck} for details). 
Therefore, as long as $(n,q)\not\in\{(3,2), (3,3)\}$ the subgroup $H$ lies 
in $\Lambda_1$.
Finally note that $\PSL_3(2)\simeq \PSL_2(7)$ and $\PSL_3(3)$ are minimal 
simple groups and so they belong to $\mathcal{L}$.

Let $S=\PSp_4(2^r)$. Note that $r\geq 2$ as $\PSp_4(2)$ is not simple. 
If $r=2$ then a direct inspection in \cite{atlas} shows that we may 
take $H\simeq \Sym{6}$. Let $r\geq 3$. If $r$ is odd, then $S$ contains 
a unique conjugacy class of maximal subgroups isomorphic to $\null^2B_2(2^r)$, 
so we are done. For $r$ even, one may take $H\simeq \PSp_4(q^{1/2})$, 
which is always a maximal subgroup of $S$, and $[H]_S=[H]_{\Aut{S}}$ 
(see \cite[Table 3.5.C]{kleidman_liebeck} or \cite[Table 8.14]{colva}).
  
Let $S=\OO_{2n}^+(q)$ with $n\geq 4$. If $n\geq 5$, we may take $H=P_{\hat{1}}$. This 
coincides with the stabilizer of an isotropic point. $H$ is a maximal subgroup of $S$ 
that extends to $\Aut{S}$ and such that $H'$ is a central extension of 
$\OO_{2(n-1)}^+(q)$ 
(see \cite[Table 3.5.E]{kleidman_liebeck} or \cite[Table 8]{colva}), thus 
$H\in\Lambda_1$.\\
For $S=\OO_{8}^+(q)$ we refer to \cite{kleidman_POmega_8}. 
If $q>2$ then maximal parabolic subgroups $P_{\hat{2}}$ that correspond to the second 
node of the Dynking diagram form a unique $S$-class of maximal subgroups of $S$, 
which is $\Aut{S}$-invariant. Such subgroups are
stabilizers of totaly singular $2$-subspaces and are denoted by $R_{s2}$ in 
\cite{kleidman_POmega_8}. Their isomorphism type is 
$[q^9].(\frac{1}{d}\mathrm{GL}_2(q)\times \Omega_4^+(q)).d$, 
where $d=(2,q-1)$ (\cite[Prop. 2.2.2]{kleidman_POmega_8}), thus as long as $q>3$, 
$H=P_{\hat{2}}$ satisfies our requirements.\\ 
If $S=\OO_8^+(3)$, we may take $H$ to be an $N_2$-group in the 
notation of \cite[Section 3.2]{kleidman_POmega_8}. Then 
by \cite[Prop. 3.2.3]{kleidman_POmega_8}, $H\simeq \PSL_3(3):2$ and 
it extends from $S$ to $\Aut{S}$ . Moreover $H$ is 
self-normalizing in $S$, as it is a maximal member in the class $\mathcal{C}_2$ 
(\cite[Prop. 4.2.1 and Table III]{kleidman_POmega_8}).\\ 
Finally consider the case $S=\OO_8^+(2)$. According to the 
\cite{atlas} (or again \cite[Section 3.2]{kleidman_POmega_8}) 
we may take $H$ a subgroup isomorphic to $\PSU_3(3):2$. Such subgroup
extends to $\Aut{S}$ and is self-normalizing, as the proper overgroups of these are 
isomorphic to $\PSp_6(2)$.

Let $S=G_2(3^r)$. If $r$ is odd, let $H$ be the centralizer in $S$ of a graph 
involution, $\phi_2$. Then  $H$ is a maximal subgroup of $S$, isomorphic to $^2G_2(q)$ 
and such that $[H]_S=[H]_{\Aut{S}}$ 
(see \cite[Lemma 1.5.5 and Theorem B]{kleidman_G_2}). 
For $r$ even one can take the centralizer of a field automorphism of order two, then 
$H$ is still a maximal subgroup that extends to $\Aut{S}$ and it is isomorphic to 
$G_2(q^{1/2})$.

Let $S=F_4(2^r)$.  
According to \cite[Table 5.1]{liebeck_saxl_seitz}, $S$ admidts a unique 
conjugacy class of subgroups (of maximal rank in $\Aut{S}$) isomorphic to 
$\textrm{Sp}_4(q^2).2$. These subgroups are not maximal in $S$ but they are 
self-normalizing. If $H$ is one of these then $H=N_S(S_1(H))$ and we know 
that $M:=N_{\Aut{S}}(S_1(H))$ is maximal in $\Aut{S}$. Then $H$ is normal 
in $M$ and so $M=M_{\Aut{S}}(H)$ forcing $H=M\cap S=N_S(H)$.

Let $S=E_6(q)$. We can take as $H$ either a maximal parabolic subgroup that 
corresponds to the middle node in the Dynkin diagram or the one below it. \\

\noindent
{\it Case $S$ sporadic or $ ^2F_4(2)'$.} \\
In Table \ref{table_sporadics} we collect our choice for the subgroup $H$
when $S$ is one of the $26$ sporadic simple groups or the Tits group 
$\null^2F_4(2)'$. In each case, 
$H$ is maximal in $S$. The validity of Table \ref{table_sporadics} can be checked 
in the Atlas (\cite{atlas}) and by the use of \cite{wilson} for what concerns the 
Monster group $M$. 

\begin{table}[htbp]
\caption{Case $S$ is sporadic or $^2F_4(2)'$.}
\label{table_sporadics}
\vspace{-0.7cm}
\[\begin{array}{|l|l||l|l||l|l|}
\hline
S  &   H    & S &  H    &  S & H  \cr
\hline
M_{11}    & M_{10}       & M_{24}  & M_{23}     & HN & \Alt{12}   \cr
M_{12}    & \PSL_2(11)   & M^{c}L  & \PSU_4(3)  & Ly & G_2(5)     \cr
J_{1}     & \PSL_2(11)   & He      & 3^{\cdot}\Sym{7}   & Th & M_{10}   \cr
M_{22}    & \PSL_2(11)   & Ru      & ^2F_4(2)  & Fi_{23}& 2^{\cdot}Fi_{22}  \cr
J_2       & \PSU_3(3)    & Suz     & G_2(4)    & Co_1 & Co_2      \cr
M_{23}    & M_{22}       & O'N     & J_1       & J_4  & 2^{11}:M_{24}  \cr
^2F_4(2)' &\PSL_2(25)    & Co_3    & M^{c}L:2  & Fi_{24}' & Fi_{23}  \cr
HS        & M_{22}       & Co_2    & M^{c}L    & B  & Th   \cr
J_3       & \PSL_2(17)   & Fi_{22} & M_{12}    & M  & \PSL_2(29):2  \cr
\hline
\end{array}\]
\end{table}
\noindent
The proof is now complete.
\end{proof}

\noindent 
The following fact will be needed in the last section.

\begin{lemma}\label{simple2} 
Let $S$ be a simple group belonging to $\mathcal{L}$.
Then there exists a $2$-nilpotent subgroup $D$ of $S$, with 
$O_2(D)=1$ and $N_S(D)=D$ that extends to $\Aut{S}$.
\end{lemma}
\begin{proof}
We treat the different cases separately.

{\it Case $S\simeq \PSL_2(2^r)$, with $r$ a prime.}

\noindent We may take $D$ to be a dihedral group of oder $2(q+1)$. Then $D$ is 
a maximal subgroup of $S$ and it extends to $\Aut{S}$, being the 
normalizer of maximal torus of order $q+1$ (see \cite[Theorem 6.5.1]{gorenstein}).  

{\it Case $S\simeq \PSL_2(3^r)$, with $r$ a prime.}

\noindent 
When $r$ is odd, $q=3^r\equiv -1 \pmod 4$ and so $\frac{q-1}{2}$ is odd. 
Then take $D$ the normalizer of a split torus of order $\frac{q-1}{2}$. This is  a 
maximal subgroup of $S$ isomorphic to a dihedral group, that extends to $\Aut{S}$.
(see \cite[Theorem 6.5.1]{gorenstein}).
When $r=2$, take $D$ the normalizer of a Sylow $3$-subgroup, then $D\simeq 3^2:4$ is a 
maximal subgroup that extends to $\Aut{S}$ (see \cite{atlas}).

{\it Case $S\simeq \PSL_2(p^{2^a})$, with $p$ an odd prime and $a\geq 0$.}

\noindent 
If $q=p^{2^a}\neq 5, 7, 11$ then take $D$ a dihedral subgroup of order $q+\epsilon$,
where $\epsilon =\pm 1$ and $q\equiv \epsilon \pmod 4$. Then $D$ is maximal in $S$ 
and it extends to $\Aut{S}$ (see \cite[Table 8.1]{colva}). If $q=5$, then 
$S\simeq \Alt{5}$ and one may take as $D$ a dihedral subgroup of order $10$. 
If $q=7$ the group $\PSL_2(7)\simeq \PSL_3(2)$. In this case, take $D$ the stabilizer 
in $\PSL_3(2)$ of a direct decomposition $V=U\oplus W$ of the natural $3$-dimensional 
module $V$ by a $1$-dimensional $U$ and $2$-dimensional $W$. Then $D$ is isomorphic to 
a dihedral group of order $6$, it is self-normalizing in $S$ and it extends to 
$\Aut{S}$. If $q=11$ take $D$ the normalizer in $S$ of Sylow $5$-subgroup, this 
self-normalizing subgroup extends to $\Aut{S}$ (see \cite{atlas}).

{\it Case $S\simeq \PSL_3(3)$.}

\noindent 
Take $D$ to be a Borel subgroup of $S$. This is a self-normalizing subgroup of $S$ 
isomorphic to $P:2^2$, where $P$ is extraspecial of order $27$ and exponent $3$. 

{\it Case $S\simeq \null^2B_2(2^r)$, where $r$ is an odd prime.}

\noindent 
We may take as $D$ a dihedral subgroup of order $2(q-1)$. This is a maximal subgroup 
of $S$ and, since $S$ has a unique conjugacy class of subgroups of order $q-1$ (see 
\cite{suzuki}), it extends to $\Aut{S}$. 
\end{proof}

\section{$n$-rarefied subgroups}
We start by defining some particular members of the class $\Lambda_n$.

\begin{defi}\label{minimal} 
A group $G$ of nonsolvable length $n\geq 1$ will be said to be \emph{$n$-rarefied} if 
the following conditions hold:
\begin{enumerate}
\item $S_i(G)/R_i(G)$ is the unique  minimal normal subgroup of $G/R_i(G)$ for all 
$i=1,2, \dots, n$;
\item the components of $S_i(G)/R_i(G)$ are simple groups in $\mathcal L$ for all 
$i=1,2, \dots, n$;
\item $R_1(G)=\Phi(G)$ and $R_{i+1}(G)/S_i(G)=\Phi(G/S_i(G))$ for all 
$i=1, \dots, n-1$.
\end{enumerate}
\end{defi}
Note that every $n$-rarefied group $G$ is perfect with last $RS$-factor, 
$G/R_{n}(G)$, a simple group.\\

The aim of this section is to show that each group in $\Lambda_n$ contains an 
$n$-rarefied subgroup (Theorem \ref{main}). The existence of  $n$-rarefied 
subgroups can be used to reduce, in certain cases, the study of a question 
about $\Lambda_n$-groups to the case of $n$-rarefied groups. However, if the 
problem we are dealing with, is not about abstract group but concerns groups 
acting on some structure, then this reduction Theorem may not be sufficient. 
In the last part of this section, we derive a reduction argument 
that covers the case of permutation and linear groups (Proposition 
\ref{nearly faithful}).\\

The class of rarefied groups is closed by quotients.
\begin{prop}\label{quotient} 
Let $G$ be an $n$-rarefied group and $N$ a normal subgroup of $G$. 
Then $G/N$ is $m$-rarefied for some $m\leq n$.
\end{prop}
\begin{proof} Clearly we may assume that $N$ is not contained in $\Phi(G)$, otherwise 
the claim is plainly true. We make induction on the order of $G$.  
Assume the claim true for groups of order smaller than $\left|G\right|$ 
and choose $N$ not contained in $\Phi(G)$.
If $\Phi(G)=1$, then $N$ must contain $S_1(G)$, which is the only minimal 
normal subgroup of $G$. Since $G/S_1(G)$ is $(n-1)$-rarefied, we apply induction 
to $\ov{N}=N/S_1(G)$ as a normal subgroup of $\ov{G}=G/S_1(G)$, getting 
that $\ov{G}/\ov{N}\simeq G/N$ is $m$-rarefied for some $m\leq n-1<n$.\\
Therefore we suppose $\Phi(G)\not=1$.  If $A=N\cap \Phi(G)\not=1$, 
we apply the inductive hypothesis to $N/A$ as a normal subgroup of $G/A$, 
and conclude as in the above paragraph. We are then left with the case 
$\Phi(G)\cap N=1$. In this situation $N$ has trivial solvable radical, 
hence $B=S_1(N)$ is semisimple and from this it follows that $S_1(G)=B\Phi(G)$.
If we can prove that $G/B$ is $(n-1)$-rarefied, the claim will follow as before. 
The first remark we should make is that $R_2(G)/B$ is contained in $R/B$, 
the solvable radical of $G/B$, since $S_1(G)/B$ is solvable. 
On the other hand $R^{(d)}\leq B\leq S_1(G)$ for some $d$, 
showing that $R\leq R_2(G)$, from which  $R=R_2(G)$ follows. 
Hence the terms of the $RS$-series of $G/B$ are 
the images of the terms of the $RS$-series of $G$, 
starting with $R_2(G)$. Thus the only thing that remains to be proved, is that 
$R_2(G)/B$ is the Frattini subgroup of $G/B$. Recall that $R_2(G)/S_1(G)$ 
is the Frattini subgroup of $G/S_1(G)$.
Set $F/B$ for the Frattini subgroup of $G/B$ and pick $M/B$ a maximal 
subgroup of $G/B$. Since $M$ is maximal in $G$ it must contain $\Phi(G)$. 
Thus $M$ contains $B\Phi(G)=S_1(G)$ and we have that 
$$
\mathcal{A}=\{M\mid B\leq M \textrm{ and } M/B \textrm{ is maximal in  } G/B\}
$$
coincides
with
$$
\mathcal{B}=\{M\mid S_1(G)\leq M \textrm{ and } M/S_1(G) \textrm{ is maximal in  } G/S_1(G)\}.
$$
Hence 
$$
F=\cap \{M\mid M\in \mathcal A\}=\cap \{M\mid M\in \mathcal B\}=R_2(G)
$$
and the proof is completed. 
\end{proof}

The next Lemma says that for the class of $n$-rarefied 
groups the nonsolvable length behaves well.
\begin{lemma}\label{lambda_min}
Let $G$ be an $n$-rarefied group and $N\nor G$. Then
$\lambda(G)=\lambda(N)+\lambda(G/N).$
\end{lemma}
\begin{proof} Since $\lambda(N)=\lambda(N\Phi(G))$, we assume $\Phi(G)\leq N$.
When $\Phi(G)\not=1$, an obvious inductive argument gives the claim, 
since $\lambda(G)=\lambda(G/\Phi(G))$ and $G/\Phi(G)$ is still $n$-rarefied.
Thus we assume $\Phi(G)=1$ which implies $S_1(G)\leq N$, unless $N=1$, 
a case that we do not need to consider.
Since $S_1(G)=S_1(N)$, we apply induction on $N/S_1(G)$ as a subgroup 
of $G/S_1(G)$, getting
$$
\lambda(G/S_1(G))=\lambda ( N/S_1(G))+\lambda(G/N)
$$
and the claim follows because $\lambda(G/S_1(G))=\lambda(G)-1$ and 
$\lambda ( N/S_1(G))=\lambda(N)-1$.
\end{proof}

Our first step towards Theorem \ref{main} is to show that every 
group in $\Lambda_n$ has a subgroup whose $RS$-series satisfies some restrictions.

\begin{lemma}\label{mono} Let $G$ be a group in $\Lambda_n$. 
Then there exists a subgroup $H$ of $G$ such that $H\in \Lambda_n$ and 
$S_i(H)/R_i(H)$ is the unique minimal normal subgroup of $H/R_i(H)$, for all 
$i=1, \dots, n$.
\end{lemma}
\begin{proof} We prove the claim by induction on $n$. 
If $n=1$ the section $S_1(G)/R_1(G)$ is a direct product of simple groups. 
Choose $H$ such that $R_1(G)\leq H$ and $H/R_1(G)$ is one of the simple direct 
factors of $S_1(G)/R_1(G)$. Thus $R_1(H)=R_1(G)$ and the claim holds.

Assume $n>1$ and the claim true for groups in $\Lambda_{n-1}$. If the claim does 
not hold in $\Lambda_n$, choose a counterexample $G\in \Lambda_n$ of 
minimal order. If $R_1(G)\not=1$, the group $G/R_1(G)$ is still in $\Lambda_n$ 
(by Lemma \ref{trivial}) but its order is strictly smaller than $\vert G\vert$. 
There is therefore a subgroup $H$ of $G$ such that $H/R_1(G)$ belongs to 
$\Lambda_n$ and has the required property. Clearly $R_1(G)\leq R_1(H)$, so that 
the preimages of the terms of the $RS$-series of $H/R_1(G)$ are the terms of the 
$RS$-series of $H$, contradicting the fact that $G$ was a counterexample. Hence 
$R_1(G)=1$.

The group $G/S_1(G)$ belongs to $\Lambda_{n-1}$, so that there exists 
$\overline{K}=K/S_1(G)\leq G/S_1(G)$ such that $\overline{K}\in \Lambda_{n-1}$ 
and $S_i(\overline{K})/R_i(\overline{K})$ is the unique minimal normal subgroup of 
$\overline{K}/R_i(\overline{K})$ for all $i=1, \dots , n-1$. Since 
$[R_1(K),S_1(G)]\leq R_1(K)\cap S_1(G)=1$, the subgroup $R_1(K)$ lies in $C_G(S_1(G))$ 
which is trivial by Lemma \ref{trivial}(4); therefore the $RS$-series of $K$ starts 
with $S_1(K)$. Of course $S_1(G)\leq S_1(K)$ and,  being $S_1(K)$ semisimple, if 
$S_1(G)\not = S_1(K)$ we have a direct decomposition $S_1(K)=S_1(G)\times C$ with 
$C\not=1$. But then  $C$ is contained in $C_G(S_1(G))=1$, a contradiction. As a 
consequence we have that $K$ belongs to $\Lambda_n$ and, by minimality of $G$,  
$G=K$.

It is then possible to decompose $S_1(G)$ as $S_1(G)=T_1\times \cdots T_r$, where each 
$T_i$ is a minimal normal subgroup of $G$. Since $G$ is a counterexample, $r$ is at 
least $2$. For every $i=1,2,\ldots, r$, set $C_i=C_G(T_i)$. If none of the 
$C_i$ is contained in $R_2(G)$, for $i=1, \ldots , r$, then each subgroup $C_iR_2(G)$ 
must contain $S_2(G)$, because $S_2(G)/R_2(G)$ is the only minimal normal subgroup of 
$G/R_2(G)$. Therefore, writing $S$ for $S_2(G)$, we have  
$$S=S\cap C_iR_2(G)=(S\cap C_i)R_2(G),$$
for all $i=1,2,\ldots,r$.
Thus
$$S'=[S,S]=[(S\cap C_1)R_2(G),(S\cap C_2)R_2(G)]\leq R_2(G)(C_1\cap C_2).$$
In this way we prove that
$$ \gamma_r(S)\leq  R_2(G)(\cap_{i=1}^rC_i). $$
However $\bigcap_{i=1}^rC_i=1$, showing that $S_2(G)/R_2(G)$ is nilpotent of class 
at most $r-1$. This can happen only if $S_2(G)/R_2(G)=1$ which, in turn, implies 
$n=1$, a contradiction. We can therefore assume, without loss of generality, that 
$C=C_G(T_1)\leq R_2(G)$. If $T=\prod_{i\not=1}T_i$, set $\overline{G}=G/T$.
We claim that $\overline{G}\in \Lambda_n$.

First of all we prove that $R_1(\overline{G})=C/T$. Write $R_1(\overline{G})=U/T$.
Being $T_1$ a minimal normal subgroup of $G$ and $U/T$  solvable, we have 
$T_1\cap U=1$. Therefore $[T_1,U]=1$, thus $U\leq C$. 
On the other hand, we know that $C\leq 
R_2(G) $, hence $C S_1(G)/S_1(G)$ is solvable and, for some $d$, ${C}^{(d)}\leq S_1(G)
\cap C=T$. This shows that $C/T$ is solvable or, in other words, that $C\leq U$ thus 
proving that $R_1(\overline{G})=C/T$. In order to prove our claim, we shall show that 
$S_1(\overline{G})=S_1(G)C/T$. What we need is to identify the socle of 
$\overline{G}/R_1(\overline{G})$ and, since this group is isomorphic to $G/C$, 
we prefer to work in this quotient of $G$. Of course $S_1(G)C/C$ is contained in the 
socle of $G/C$. The subgroup $S_1(G)C/C$ is normal so, in particular, it is normal in 
$S_1(G/C)$, which is semisimple as $G/C\simeq \overline{G}/R_1(\overline{G})$. 
This means that we can write $S_1(G/C)= S_1(G)C/C\times L/C$, where $L/C$ is normal 
and, if non  trivial, it is the direct product of nonabelian simple groups. Taking 
commutators we get $[T_1,L]\leq T_1\cap C=1$ thus showing that $L\leq C$. Hence 
$S_1(G/C)=S_1(G)C/C$ and from this it follows that
$$
\frac{\overline{G}}{S_1(\overline{G})}\simeq 
\frac{\overline{G}/R_1(\overline{G})}{S_1(\overline{G})/R_1(\overline{G})}\simeq
\frac{G/C}{S(G/C)}\simeq
\frac{G/C}{S_1(G)C/C}\simeq 
\frac{G}{S_1(G)C}.
$$
Now, the subgroup $S_1(G)C$ is contained in $R_2(G)$ so that $G/S_1(G)C$ has an image 
isomorphic to $G/R_2(G)\in \Lambda_{n-1}$.
This proves that $\overline{G}$ belongs to $\Lambda_n$.

Let $H$ be a supplement to $T$ in $G$ of minimal order. Of course, $H$ is a proper 
subgroup of $G$ (by an easy application of Frattini argument for instance). 
The subgroup $D=H\cap T$ is then nilpotent. For, if $Q$ is a nontrivial Sylow 
$p$-subgroup of $D$, the Frattini argument gives $H=D N_H(Q)$, so that $N_H(Q)$ 
is a supplement to $T$ in $G$. By the minimal choice of $H$,
we have $N_H(Q)=H$ and, since all the Sylow subgroups of $D$ are normal, $D$ 
is nilpotent. Clearly $D\leq R_1(H)$ and
$$ \frac{H}{D}=\frac{H}{H\cap T} \simeq \frac{HT}{T}=\frac{G}{T}=\overline{G}. $$
Therefore we have
$$\frac{H}{R_1(H)}\simeq \frac{H/D}{R_1(H)/D}\simeq  
\frac{\overline{G}}{R_1(\overline{G})}, $$
showing that $H$ is in $\Lambda_n$ and has the desidered property. 
This contradiction completes the proof of the Lemma. 
\end{proof}

\bigskip

In the course of our analysis we have to treat the following situation. 
A finite group $G$ has a unique minimal normal subgroup $T$ which is 
nonabelian, and so a direct product of simple groups $S_i$, for 
$i =1, \ldots, r$ all isomorphic say to $S$. Since $G$ acts transitively 
on the set $\{S_i\vert i=1,\ldots, r\}$, 
for each $i=1, \ldots, r$ we fix $g_i\in G$ in such a way that $S_i = S_1^{g_i}$,  
and choose $g_1 = 1$. We also fix an isomorphism $\alpha$ 
from $S$ to $S_1$ and, given any subgroup $H$ of $S$, we let
$$H^* = H_1 \times H_2 \times \ldots \times H_r$$
the subgroup of $T$ such that for every $i = 1,\ldots, r$, $H_i = (H^{\alpha})^{g_i}$. 
In particular, $T=S^*$.

\begin{lemma}\label{extend}
Let $G$ be a finite group and, with the above notation, suppose 
that $S^*$ is the unique minimal normal subgroup of $G$ which is nonabelian. 
If $H$ is a proper subgroup of $S$ that extends from $S$ to $\Aut{S}$, then $H^*$
extends from $S^*$ to $G$.
\end{lemma}
\begin{proof} 
For every $x\in G$ denote by $\sigma_x \in \Sym{r}$ the permutation induced by $x$ 
on the set $\{S_i\vert i\leq r \}$, so that, in our notation, for every 
$i=1,\ldots, r$:
$$  S_1^{g_ix}=S_i^x=S_{\sigma_x(i)}=S_1^{g_{\sigma_x(i)}}.  $$
Now for every $i$, the component $S_i^x$ contains both the subgroups 
$H_i^x$ and $H_{\sigma_x(i)}$, and we claim that these subgroups
are $S_i^x$-conjugate. 
Note that $g_ix(g_{\sigma_x(i)})^{-1}\in N_G(S_1)$ and since $H_1$ 
extends from $S_1$ to 
$\Aut{S_1}$, there exists an element $s_i\in S_1$ such that
$$H_1^{g_ix(g_{\sigma_x(i)})^{-1}}=H_1^{s_i},$$
equivalently:
$$H_i^x=H_1^{s_ig_{\sigma_x(i)}}=
\left(H_{\sigma_x(i)}\right)^{(g_{\sigma_x(i)})^{-1}s_ig_{\sigma_x(i)}},$$
which proves our claim since 
$(g_{\sigma_x(i)})^{-1}s_ig_{\sigma_x(i)}\in S_1^{g_{\sigma_x(i)}}=S_i^x$. 
If we set
$$t=\prod_{i=1}^r \left((g_{\sigma_x(i)})^{-1}s_ig_{\sigma_x(i)}\right)$$
then $t\in S^*$ and we have that 
$$(H^*)^x=(H^*)^{t},$$
therefore $x\in S^*N_G(H^*)$, which proves the Lemma.
\end{proof}

\begin{lemma}\label{normalizer} Let $G$ be a group in $\Lambda_n$.  
Then $G$ contains a subgroup $M$ belonging to $\Lambda_n$, such 
that the components of $S_1(M)/R_1(M)$ are in $\mathcal{L}$.
\end{lemma}
\begin{proof}
Let $G$ be a minimal counterexample, which, by Lemma \ref{mono}, we assume 
having $S_i(G)/R_i(G)$ as unique minimal normal subgroup of $G/R_i(G)$, for all 
$i=1, \dots, n$. In particular, write
$$S_1(G)=S^*=S_1\times S_2\times\ldots\times S_r,$$ 
where each $S_i\simeq S$ is simple and $r\geq 1$.\\
Note that $S\not \in\mathcal{L}$, otherwise we may take $M=G$ and $G$ is no
more a counterexample. \\ 
By Lemma \ref{simple1}, we choose a proper self-normalizing subgroup $H$ of $S$ 
which belongs to $\Lambda_1$ and such that it extends from $S$ to $\Aut{S}$. 
By Lemma \ref{extend}, the $\Lambda_1$-subgroup $H^*$ of $S^*$ may be 
extended to an $G$, i.e. $G=S_1(G)M$, where $M=N_G(H^*)$ is a proper subgroup 
of $G$. We show that $M\in\Lambda_n$, then the contradiction 
will follow by the minimal choice of $G$.\\
%
%
%
\noindent 
We set $A=R_1(H^*)=(R_1(H))^*$, $B=S_1(H^*)=(S_1(H))^*$ and $C=C_M(B/A)$ and we 
proceed by steps.\\
{\it 1)} $A= H^*\cap R_1(M)$ and $B=H^*\cap S_1(M)$.\\
Since $H^*\triangleleft M$, by Lemma \ref{trivial} we have that 
$A \leq  H^*\cap R_1(M)$ and $B\leq H^*\cap S_1(M)$. Conversely, both 
$H^*\cap R_1(M)$ and $H^*\cap S_1(M)$ are normal in $H^*$, thus by 
Lemma \ref{trivial} again, we obtain 
$$H^*\cap R_1(M)= R_1(H^*\cap R_1(M))\leq R_1(H^*)=A$$
and 
$$H^*\cap S_1(M)= S_1(H^*\cap S_1(M))\leq S_1(H^*)=B.$$
{\it 2)} $C$ is solvable.\\
Note that $C$ normalizes every component $S_i$ of $G$.
This is trivial when $r=1$, while if $r>1$ let $x\in C$ and assume that $x$ 
moves $S_i$ to $S_j$ (with $j\neq i$), then $x$  
moves $H_i A/A$ onto $H_j A/A$ and then $B/A$ is no more centralized by $x$, 
which is a contradiction. Hence every element of $C=C_M(B/A)$ induces an 
automorphism of each $S_i$ and in particular the group 
$CS^*/S^*\simeq C/C\cap S^*$ is solvable (being a subdirect product of 
outer automorphism group of a simple group).
Also, $C\cap S^*\leq N_{S^*}(H^*)=H^*$, thus $(C\cap S^*)A/A\leq C_{H^*/A}(B/A)$,
which is trivial by Lemma \ref{trivial}(4). 
Thus $C\cap S^*\leq A$ and $C$ is solvable. \\
{\it 3)} $BR_1(M)=S_1(M)$.\\
As $H^*\triangleleft M$, by Lemma \ref{trivial} we have that $BR_1(M) \leq S_1(M)$.
Assume by contradiction that the inclusion is proper and choose a direct 
complement $D/R_1(M)$ of $BR_1(M)/R_1(M)$ in $S_1(M)/R_1(M)$. Then 
$$[B,D]\leq R_1(M)\cap B\leq R_1(M)\cap H^*=A,$$
by step {\it 1)}. This means that $D\leq C$ and so by step {\it 2)}, 
$D$ is solvable, which is a contradiction.  \\
{\it 4)}  $M\in \Lambda_n$.\\ 
Note first that being $H$ self-normalizing in $S$, the subgroup $H^*$ is 
self-normalizing in $S^*$. Hence
 $$\frac{M}{H^*}= \frac{N_G(H^*)}{N_{S^*}(H^*)}=
 \frac{N_G(H^*)}{S^*\cap N_{G}(H^*)}
 \simeq \frac{G}{S^*}$$
and therefore it belongs to $\Lambda_{n-1}$. Moreover,
as $H^*=R_2(H^*)$ and $H^*\nor M$, we have $H^*\leq R_2(M)$. Also
by step {\it 3)} that $H^*S_1(M)=H^*R_1(M)\leq R_2(M)$ and hence 
$H^*S_1(M)/H^*$, being solvable and normal in $M/H^*$, lies in the 
solvable radical of $M/H^*$, which is $R_2(M)/H^*$. This with $
M/H^*\in\Lambda_{n-1}$ implies that $M/R_2(M)\in \Lambda_{n-1}$ and 
therefore $M\in\Lambda_{n}$. 
\end{proof}

We are now in a position to prove our main result. 

\noindent
{\sl{Proof of Theorem \ref{main}}.} Let $G\in\Lambda_n$ be a counterexample such that $\vert G\vert+n$ is minimal.
By Lemma \ref{mono} we may assume that $S_i(G)/R_i(G)$ is the unique minimal normal 
subgroup of $G/R_i(G)$, for every $i=1,2,\ldots, n$. 

Assume that $R_1(G)\not=1$. Then by the minimal choice, the group 
$\overline{G}=G/R_1(G)$ contains $n$-rarefied subgroup 
$\overline{K}=K/R_1(G)$. Note that being $R_1(G)\leq R_1(K)$ we have that 
$R_i(G)\leq R_i(K)$ and $S_i(G)\leq S_i(K)$ for each $i=1,\ldots, n$ and therefore 
$$\frac{R_{i+1}(\overline{K})}{S_i(\overline{K})}\simeq \frac{R_{i+1}(K)}{S_i(K)}\quad
\textrm{ and }\quad
\frac{S_{i}(\overline{K})}{R_i(\overline{K})}\simeq \frac{S_{i}(K)}{R_i(K)}.
$$
In particular, if $R_1(K)\leq \Phi(K)$, then $K$ is an $n$-rarefied 
subgroup of $G$, which is a contradiction. Hence we have that 
$R_1(K)\not\leq \Phi(K)$. Let $M$ be a maximal subgroup of $K$ such that 
$K=R_1(K)M$. Since 
$$\frac{R_1(K)}{R_1(G)}=R_1(\overline{K})=\Phi(\overline{K}),$$
we have that $R_1(G)\not\leq M$ and therefore $K=R_1(G)M$. 
Call $H$ a minimal supplement of $R_1(G)$ in $K$. 
Note that $R_1(H)\leq \Phi(H)$, otherwise taken $L$ a maximal subgroup 
of $H$ that supplements $R_1(H)$ in $H$, since 
$R_1(H)/H\cap R_1(G)=\Phi(H/H\cap R_1(G))$, 
$L$ does not contain $H\cap R_1(G)$, forcing $H=(H\cap R_1(G))L$ and therefore 
$K=R_1(G)L$, which contradicts the minimal choice of $H$. 
Therefore $R_1(H)\leq \Phi(H)$ and, since $H\cap R_1(G)\leq R_1(H)$ 
and $H/H\cap R_1(G)\simeq \overline{K}$, it easily 
follows that $H$ is an $n$-rarefied subgroup of $G$, which is a 
contradiction.

Therefore $R_1(G)=1$. If $n=1$, our minimal choice implies that $S_1(G)=G$ is a simple 
group not belonging to $\mathcal{L}$. Take $H$ a proper subgroup $G$ satisfying the 
properties of Lemma \ref{simple1}. As $H\in\Lambda_1$ and 
$\vert H\vert<\vert G\vert$, then $H$ (and therefore $G$ too) contains an 
$1$-rarefied subgroup, which is a contradiction. Therefore $n>1$. 
By inductive assumption, $G/S_1(G)$ contains a $(n-1)$-rarefied subgroup  
$H/S_1(G)$. Of course, $S_1(G)\leq S_1(H)$. If $S_1(G)<S_1(H)$ let $B/R_1(H)$ 
be a complement of $S_1(G)R_1(H)/R_1(H)$ in $S_1(H)/R_1(H)$. Then 
$$[S_1(G),B]\leq S_1(G)\cap R_1(H)\leq R_1(G)=1,$$
which implies that $C_G(S_1(G))\ne 1$ and this contradicts Lemma \ref{trivial} (4.). 
It follows that $S_1(G)=S_1(H)$ and, since $H/S_1(H)\in \Lambda_{n-1}$, 
$H$ belongs to $\Lambda_{n}$. Again the minimality of $G$ implies that $G=H$. 
Remember that $S_1(G)$ is the unique minimal normal subgroup of $G$, therefore if 
its components lie in $\mathcal{L}$, then $G$ is $n$-rarefied, which is \
a contradiction. We may therefore apply Lemma \ref{normalizer} to find a subgroup $M$ 
of $G$ having nonsolvable length $n$ and such that the components of $S_1(M)/R_1(M)$
are in $\mathcal{L}$. This last condition implies that $M<G$. By induction 
$M$, and therefore $G$, contains an $n$-rarefied subgroup, which is the last 
contradiction.\qed

\bigskip 
 
The next result is useful when  dealing with permutation or linear groups.

\begin{prop}\label{nearly faithful} 
Let $G$ be a group in $\Lambda_n$ and $H$ an $n$-rarefied 
subgroup of $G$. Then 
\begin{enumerate}
\item if $G$  acts faithfully on the set $\Omega$, there is an  $H$-orbit 
$\Delta$, such that $H/C_H(\Delta)$ is in $\Lambda_n$;
\item if $G$ acts faithfully of the finite dimensional $\mathbb F$-vector space $V$, 
there exist an $H$-invariant irreducible section $W$ of $V$, such that $H/C_H(W)$ 
is in $\Lambda_n$.
\end{enumerate}
\end{prop}
\begin{proof} Let $\{N_i \mid i=1, \dots , l\}$ be the set of kernels 
of the action of $H$ on its orbits, we have $\cap_{i=1}^l N_i=1$. 
By Lemma \ref{subdirect3} 
we must have $H/N_i\in \Lambda_n$ for at least one index $i\in\{1,\ldots,l\}$, 
proving the claim.
When $G\leq \mathrm{GL}(V)$, the same argument works, if the $N_i$ are the 
kernels of the action of $H$ on the factors of an $H$-composition series in $V$. 
\end{proof}

\section{Some applications}\label{application}

In this section we will apply the result obtained to the study of some 
questions concerning the nonsolvable length.\\	

Our first aim is to improve the bound obtained in \cite[Theorem 
1.1.(a)]{khukhro}, where the authors showed that $\lambda(G)\leq 2L_2(G)+1$, 
being 
$$L_2(G)=\max\{l_2(H)\mid H \textrm{ is a solvable subgroup of } G\}$$
and $l_2(H)$ the minimal number of $2$-factors in a  $2\,2'$-series 
of the solvable group $H$.

\begin{lemma}\label{bound} Let $G$ be an $n$-rarefied group with trivial 
Frattini subgroup. 
Then $G$ contains a solvable subgroup $H$ such that $l_2(H/O_2(H))\geq n$.
\end{lemma}

\begin{proof} We prove the Lemma by induction on $n$. 

If $n=1$ then $G$ is a simple group in $\mathcal{L}$ and an inspection of the 
sugbroups of $G$ reveals that the claim holds (see e.g. 
\cite[Section 6.5]{gorenstein}).

Assume the claim true for $n-1$ and choose $G\in \Lambda_n$ with $\Phi(G)=1$.
Write  $$ S_1(G)=S^*=\prod_{i}S_i $$
each $S_i\simeq S$ a simple group in $\mathcal L$. The group $G$ permutes the 
components and $G=S^*\cdot X$, where $X$ acts on the sets of components and, 
according to Lemma \ref{simple2}, it is chosen such that $X\leq N_G(D^*)$, where 
$D^*$ is a $2$-nilpotent subgroup of $S^*$ with $O_2(D^*)=1$ and 
$N_{S^*}(D^*)=D^*$ (we used Lemma \ref{extend}). \\
Note that $X/X\cap S^*$ is a $(n-1)$-rarefied group, therefore 
if we set $F/X\cap S^*=R_1(X/X\cap S^*)=\Phi(X/X\cap S^*)$, by inductive 
hypothesis we have that there exists a solvable subgroup $L/F\leq X/F$ 
such that, if $A/F=O_2(L/F)$, 
$$l_2\left(\frac{L/F}{A/F}\right)=l_2(L/A)\geq n-1.$$
Let also $C$ be the kernel of the permutation action of $X$ on $\{S_i\}$, 
$C=\bigcap_i N_X(S_i)$. 
Then $C/C\cap S^*$ embeds, as a subdirect  product, into
$\prod_{i}\mathrm{Out}(S_i)$, which is solvable. Therefore $C\leq F$.
We set $H=D^*L$ and we claim that $l_2(H/O_2(H))\geq n$.\\
Call $B/O_2(H)=O_{2'\, 2}(H/O_2(H))$. Then 
\begin{equation}\label{eq:bound}
l_2(H/O_2(H))=l_2(H/B)+1.
\end{equation}
Note that 
$$\frac{H}{D^*A}=\frac{(D^*A)L}{D^*A}\simeq \frac{L}{D^*A\cap L}=
    \frac{L}{A(D^*\cap L)}=\frac{L}{A},$$
since $D^*\cap L\leq S^*\cap X\leq F\leq A$.
It follows that if we show
\begin{equation}\label{eq:bound_2}
B\leq D^*A
\end{equation}
then 
$$l_2(H/B)\geq l_2(H/D^*A)=l_2(L/F)\geq n-1$$
and so, by (\ref{eq:bound}),
$$l_2(H/O_2(H))\geq n.$$
Now, $O_2(H)$ centralizes every $U_i\leq U^*$, thus $O_2(H)\leq D^*C\leq D^*A$. 
Let $R=O_{2\, 2'}(H)=O_2(H)\rtimes T$, where $T$  exists by Shur-Zassenhauss 
and has odd order, so that $B=O_{2\, 2'\, 2}(H)=RP$ for any Sylow $2$-subgroup 
$P$ of $B$. 
Let $t$ be any element of $T$. If $S_i^t=S_j$ for $j\neq i$, then chosen a 
$2$-element $v_i\in D_i$, we have that
$$[t,v_i]=(v_i^{-1})^tv_i$$ 
is a nontrivial $2$-element of $S^*\cap R$, thus a nontrivial $2$-element of 
$S^*\cap O_2(H)$. But $S^*\cap O_2(H)\leq O_2(D^*)=1$. We have therefore that 
$T$ normalizes every component, and thus 
$T\leq D^* C\leq D^*A$, forcing $R$ itself to be contained in $D^*A$. \\
Finally note that
$B=RP\leq (D^*A)P$ 
and so $(D^*A)P=(D^*A)B$ is a normal subgroup of $H$ such that
$$\frac{(D^*A)B}{D^*A}\simeq \frac{P}{P\cap D^* A}$$
is a $2$-group. Therefore 
$$\frac{(D^*A)B}{D^*A}\leq O_2\left(\frac{H}{D^*A}\right)=
                                 O_2\left(\frac{L}{A}\right)=1,$$
proving that $B\leq D^*A$.
\end{proof}

\bigskip

The improvement of the bound obtained in \cite[Theorem 1.1]{khukhro} follows 
easily. 

\bigskip

\begin{thm}\label{bound1} Let $G$ be any finite group. Then $\lambda(G)\leq L_2(G)$.
\end{thm}

\begin{proof} The group $G$ contains an $n$-rarefied subgroup $G_0$. 
By Lemma \ref{bound}, $G_0/\Phi(G_0)$ constains a solvable subgroup 
$H/\Phi(G_0)$ with $l_2(H/\Phi(G_0))\geq n$ and $O_2(H/\Phi(G_0))=1$.
As $\Phi(G_0)$ is nilpotent, $l_2(H)\geq l_2(H/\Phi(G_0))\geq n$.
Therefore $n\leq L_2(G_0)$ and, since we clearly have $L_2(G_0)\leq L_2(G)$, 
we obtain $n\leq L_2(G)$, as claimed. 
\end{proof}

\bigskip

Our next application of Theorem \ref{main} is related to a general problem 
studied in \cite{LMS}. In that paper  the authors address the following question.
\begin{quote}
Let $\mathcal{P}$ be a group theoretical property, and $G$ a finite 
 group possessing $\mathcal{P}$. What is the minimal number $d$ such that 
 $G$ has a $d$-generated subgroup possessing $\mathcal{P}$?
\end{quote} 
%
%

The following result answer the aforementioned question for the property  
$\mathcal{P}$ consisting of having nonsolvable length $n$. 
At the same time it is an improvement of \cite[Theorem 1.1]{DS}.

\begin{thm}\label{generators}
Let $G$ be any finite group. If $G$ is not solvable then there exists a 
$2$-generator subgroup $H$ of $G$ such that $\lambda(H)=\lambda(G)$.
\end{thm}
\begin{proof} Let $n\geq 1$ and $H$ an $n$-rarefied group. We prove that 
$H$ is $2$-generated. 
It is known, as a byproduct of the classification of finite simple groups, 
that every simple group can be generated by two elements. So, if $H$ is 
$1$-rarefied, then $H/\Phi(H)$ is $2$-generated. 
Clearly $H$ is $2$-generated as well. Assume our claim true for groups of 
nonsolvable length at most $n-1$ and choose $H$ any $n$-rarefied 
group. It is harmless to assume that $\Phi(H)=1$, since any set of 
elements generating $H$ modulo $\Phi(H)$, also generates $H$.
Thus $S_1(H)$ is the unique minimal normal subgroup of $H$, and we can use 
the main result of \cite{LM} to see that the minimal number of generators 
of $H/S_1(H)$ is the same as the minimal number of generators of $H$. 
The inductive hypothesis yelds that $H$ is $2$-generated and Theorem \ref{main}
completes the proof. 
\end{proof}

\bigskip

We consider now a problem of different nature. 

For any natural number $m$ and any field $\mathbb F$, define
\begin{align*}
\lambda (m)&=\max\{\lambda (G) \mid G\leq \mathrm{Sym}(m)\} \\
\lambda_\mathbb F(m)&=
\max\{\lambda (G) \mid G\leq \mathrm{GL}(m,\mathbb F)\}
\end{align*}

We shall use Theorem \ref{main} to find  $
\lambda (m)$ and $ \lambda_\mathbb F(m)$ for all $m$ and $\mathbb F$. 
We start by establishing lower bounds for the degree of permutation 
and linear  representations of $n$-rarefied groups.

\bigskip

\begin{lemma}\label{inequality}
Let $n,k$ be natural numbers with $1\leq k<n$. Then $(n-k)5^k\geq n$.
\end{lemma}

\begin{proof} It is easy to check that the function $f(x)= (n-x)5^x$ is strictly 
increasing in the interval $[1, n-1/\log(5)]$ so that, when restricted 
to $[1,n-1]\cap \mathbb N$, it attains its minimum at $x=1$. 
Since $f(1)=5(n-1)>n$, the Lemma is proved. 
\end{proof}

\bigskip

\begin{lemma}\label{degree} Let $G$ be any finite group 
of nonsolvable length $n$ acting faithfully on the set $\Omega$. 
Then $\left| \Omega\right|\geq 5^n$.
\end{lemma}

\begin{proof}
Of course we may assume that $G$ acts transitively on $\Omega$.
Moreover by Theorem \ref{main} and Proposition \ref{nearly faithful}, $G$ can be 
considered to be $n$-rarefied.\\
We prove the claim by induction on $n$. 

Assume $n=1$. If $\Phi(G)=1$ then $G$ is a simple group in $\mathcal L$. Being the minimal degree 
of a faithful representation $5$, the claim holds. 
If $\Phi(G)\not=1$ then $G$ can not be primitive because, in this case, 
$\Phi(G)$ would be transitive and then, for every $\omega\in \Omega$, we would 
have $G=\Phi(G) G_\omega$, which is clearly impossible. Thus 
the orbits of $\Phi(G)$ form a system of non-trivial blocks for $G$. Also,
$G/\Phi(G)$ acts faithfully and transitively on the set of $\Phi(G)$-orbits and
we may therefore conclude as above. 

Assume $n>1$ and suppose the claim is true for $\Lambda_k$-groups with $1\leq k<n$.\\
Consider first the case $G$ primitive. As we have pointed out before, 
in this situation the Frattini subgroup must be trivial.
Thus $S_1(G)$ is the socle of $G$ and its components, all isomorphic to a fixed 
simple group $S\in \mathcal L$, are permuted by $G/S_1(G)$. If $K/S_1(G)$ is the 
kernel of this action, then $K/S_1(G)$ is solvable and, therefore 
$K\leq R_2(G)$ and $G/K$ is a $(n-1)$-rarefied group. 
The inductive hypothesis  shows that $S_1(G)$ has at least $5^{n-1}$ components. 
Clearly $S_1(G)$ acts transitively on $\Omega$. 
If $S_1(G)$ acts regularly, then 
$$ \left|\Omega \right|= \left|S_1(G) \right|\geq 
            \left|\mathrm{Alt}(5) \right|^{5^{n-1}}\geq 5^n. $$
When the socle is not regular we can use 
\cite[Theorem 4.6A]{dixon}.
The only possibilities, in our situations, are the following:
\begin{enumerate}
\item $S_1(G)$ acts in diagonal action, so that 
$\left|\Omega \right|=\left|S_1(G) \right|/\left| S\right|$, or 
\item $G$ is a transitive subgroup of 
$W= U\mathrm{wr}_\Gamma \mathrm{Sym}(\Gamma)$ where $U$ is primitive non-regular, 
$\Gamma$ has at least two elements, and $W$ acts in product action.
\end{enumerate}

If we are in case $1$, then the claim holds, because 
$\left|S_1(G) \right|/\left| S\right|$ is at least $ 60^{5^{n-2}}$ 
which is bigger than $5^n$.

Assume then we are in case $2$, so that the set $\Omega$ can be 
identified with $\Delta^\Gamma$, where $\Delta$ is a primitive $U$-set. 
If $B$ indicates the base subgroup of $W$, the group $G/G\cap B$ is 
isomorphic to a transitive subgroup of $\mathrm{Sym}(\Gamma)$. 
Moreover by Proposition \ref{quotient}, $G/G\cap B$ is a 
$k$-rarefied group for some $1\leq k< n$.
Fix $j\in \Gamma$ and let $N_j$ be the kernel of the projection 
from $M=G\cap B$ onto the $j$-th component of $B$. The group $M$ embeds, 
as a subdirect product, into $\prod_{i\in\Gamma}M_i$ where all 
factors $M_i= M/N_i$ are isomorphic, since $G$ acts transitively 
on $\Gamma$. If $M_j$ belongs to $\Lambda_l$, then $M$ itself is in 
$\Lambda_l$, because of Lemma \ref{subdirect3}.  From Lemma \ref{lambda_min}, 
it follows that $n=k+l$. Therefore, using Lemma \ref{nearly faithful} and 
the induction, we have that $\left|\Delta\right|\geq 5^{l}$. Thus
$$ \left| \Omega\right|=
     \left|\Delta\right|^{\left|\Gamma\right|}\geq (5^{n-k})^{5^k}\geq 5^n $$
the last inequality holds by Lemma \ref{inequality}.\\
It remains to handle the case when $G$ is imprimitive. 
Let  $\Sigma$ be the system of imprimitivity consisting of the 
$S_1(G)$-orbits and let $N$ be its stabilizer. 
Then $G/N$  acts transitively on $\Sigma$ and, if $G/N$ belongs to 
$\Lambda_k$, it is a $k$-rarefied group. Notice that, 
since $S_1(G)\leq N$, $k<n$. The inductive assumption gives 
$\left|\Sigma\right|\geq 5^k$. Using an argument similar to the one 
of the above paragraph, it is readily seen that each block in $\Sigma$ 
has size at least $5^{n-k}$. The claim follows easily. 
\end{proof}

\begin{thm}\label{lambda(m)} 
For every $m\geq 5$ we have $\lambda(m)=\lfloor \log_5(m)\rfloor$.
\end{thm}
\begin{proof}  Let $n$ be such that  $5^n\leq m< 5^{n+1}$, so that 
$n=\lfloor \log_5(m)\rfloor$. The symmetric group $\mathrm{Sym}(m)$ 
contains a subgroup $G$ isomorphic to the $n$-fold wreath product of 
$\Alt{5}$ in its natural action, acting on a set of $5^n$ points. 
Since $G$ is in $\Lambda_n$ (actually it is an $n$-rarefied 
group), it follows that $\lambda (m)\geq n=\lfloor \log_5(m)\rfloor$. 

Now let $G$ be a subgroup of $\mathrm{Sym}(m)$, whose nonsolvable 
length is $\lambda(m)$. By Theorem \ref{main} there exists 
$H\leq G$ which is a $\lambda(m)$-rarefied group and, 
using Lemma \ref{nearly faithful} and Lemma \ref{degree}, we infer 
that $m\geq 5^\lambda$. From this we get $\lambda(m)\leq \log_5(m)$. 
Therefore
$$  \lfloor \log_5(m)\rfloor \leq \lambda(m) \leq \log_5(m)  $$
and the equality $\lambda(m)=\lfloor \log_5(m)\rfloor$ follows.
\end{proof}

\begin{lemma}\label{proj} 
Let $G$ be a group with $G/\zeta(G)=\prod_{i=1}^l G_i$ where each $G_i$ 
is a finite simple nonabelian group, and $\mathbb{K}$ be 
any algebraically closed field. If the $\mathbb{K}$-vector space $V$ 
affords an irreducible  projective representation $\rho$ with 
$\ker (\rho)=\zeta(G)$, then $\dim_\mathbb F(V)\geq 2^l$.
\end{lemma}
\begin{proof}
The claim is true when $l=1$, so assume $l\geq 2$ and that 
the result holds when the number of 
factors is smaller than $l$.
Let $H\leq\mathrm{GL}(V)$ be such that $H=G^\rho Z$, where $Z$ is the center of 
$\mathrm{GL}(V)$. It is then possible to apply 
\cite[Theorem 11.20]{curtis} to $H$.
Let $K$ be such that $K/Z= G_1$, the space $V$ can be decomposed as $U\otimes W$ and 
the action of $H$ is the tensor product of two projective representations
$\sigma: K\longrightarrow \mathrm{GL}(U)$ and 
$\tau: H/K\longrightarrow \mathrm{GL}(W)$. 
It is easy to see that $\tau$ has trivial kernel, because $H/K$ is the direct 
product of simple groups and the kernel of $\rho$ is $\zeta(G)$. Since 
the number of simple factors of $H/K$ is $l-1$, induction yelds 
$\dim(W)\geq 2^{l-1}$, while $U$ has dimension at least $2$. 
Thus $\dim(V)=\dim(U)\dim(W)\geq 2\cdot 2^{l-1}=2^l$, as claimed.
\end{proof}

\begin{lemma}\label{linear degree} 
Let $G$ be any $\Lambda_n$-group acting faithfully and irreducibly 
on the $\mathbb F$-vector space $V$, where $\mathbb{F}$ is any field.  
Then  $\dim(V)\geq 2\cdot 5^{n-1}$.
\end{lemma}
\begin{proof} If $\mathbb K$ is any algebraically closed field containing $\mathbb F$, 
the group $G$ acts on $W=V\otimes_\mathbb F \mathbb K$. 
By Proposition \ref{nearly faithful} we know that $G$ has a
$\Lambda_n$-subgroup acting faithfully on a $G$-composition factor of $W$. 
A lower bound on the  $\mathbb K$-dimension of such factor would entail 
a lower bound for the $\mathbb F$-dimension of $V$.  There is therefore 
no loss of generality, if we assume that $\mathbb F$ is algebraically closed.

If $n=1$ the result is clear since $G$ can not have representations of degree $1$. 

Assume the statement true up to $n-1$.
It is harmless to assume that $G$ is $n$-rarefied, because of 
Theorem \ref{main} and Proposition \ref{nearly faithful}.

When $G$ is imprimitive, let $\mathcal B=\{W_1, \dots, W_l\}$ be a 
system of imprimitivity, and $N$ the normalizer of $\mathcal B$ in $G$. 
If $k=\lambda(N)$  we have, by Proposition \ref{quotient} 
and Lemma \ref{lambda_min}, that $G/N$ is an $(n-k)$-rarefied group.
If $k=0$ then $G/N$ is in $\Lambda_n$ and therefore $\mathcal{B}$ must 
contain at least $5^n$ elements, by  Lemma \ref{degree}. In this case $V$ 
has dimension at least $5^n>2\cdot 5^{n-1}$. Assume then that $k$ is at least $1$.
The group  $N$ is a subdirect product of the groups $N/C_N(W_i)$ 
(which are all isomorphic) hence, by Lemma \ref{subdirect3}, each $N/C_N(W_i)$ 
is in $\Lambda_{k}$. The inductive hypothesis yelds 
$\dim(W_i)\geq 2\cdot 5^{k-1}$ and Lemma \ref{degree} tells us that 
$\mathcal B$ contains at least $5^{n-k}$ elements. These two facts give
\[
\dim(V)\geq 2\cdot 5^{n-k-1}\cdot 5^k= 2\cdot 5^{n-1}.
\]

It remains to consider the case $G$ primitive. 

Assume first that $\Phi(G)$ is abelian. When this happens $\Phi(G)$ is the 
center of $G$ and $S_1(G)/\Phi(G)\simeq\prod_{i=1}^lS_i$, where each $S_i$ 
is isomorphic to a simple group $S$ in $\mathcal L$. Moreover the module $V$ 
is the direct sum of copies of a single irreducible $S_1(G)$-module. 
It is therefore enough to prove that the claimed bound holds, when $V$ is 
already irreducible as a module for $S_1(G)$. By Lemma \ref{proj} we have 
that $\dim(V)\geq 2^l$. On the other hand, we can use Lemma \ref{degree}, 
applied to the group $G/S_1(G)$ in its action on the components of 
$S_1(G)/\Phi(G)$, to see that $l\geq 5^{n-1}$. Thus
$\dim(V)\geq 2^{5^{n-1}}\geq 2\cdot 5^{n-1}$.

It remains to treat the case $\Phi(G)$ not abelian. In this case 
$\zeta(\Phi(G))=\zeta(G)$ and there exists a subgroup $A$ such that $\zeta(\Phi(G))\leq A\leq \Phi(G)$ and 
$\left|A/\zeta(\Phi(G))\right|=r^2$ 
for some $r$ dividing $\dim(V)$, see \cite[Theorem 3.3]{we}. 
 The group 
 $$
B= A/\zeta(G)\cap \zeta (\Phi(G)/\zeta(G))
 $$
 is not trivial, so there is at least one prime $p$ such that
 $\ov{P}=P/\zeta(G)$,  the $p$-Sylow of $B$, is not $1$. The group $\ov{P}$ is elementary abelian (see \cite[Theorem 3.3]{we}).  Since  
$[P, C_G(\ov{P}), C_G(\ov{P})]=1$, the group  $C_G(\ov{P})$ is nilpotent and 
therefore contained in $\Phi(G)$. On the other hand  $\Phi(G)\leq C_G(\ov{P})$, 
hence the group $\ov{G}=G/C_G(\ov{P})=G/\Phi(G)$ is an $n$-rarefied 
group contained in $\mathrm{GL}(\ov{P})$. Since $\ov{G}$ has trivial 
Frattini group, we deduce, using the first part of this proof (the imprimitive case and the  primitive case for groups with abelian Frattini group),  that the dimension of $\ov{P}$ over the field with 
$p$ elements is at least $2\cdot 5^{n-1}$. Therefore 
$$
(\dim(V))^2\geq \left| A\right|\geq p^{2\cdot 5^{n-1}}\geq 2^{2\cdot 5^{n-1}}.
$$
It follows that $\dim(V)\geq 2^{5^{n-1}}\geq 2\cdot 5^{n-1}$. The proof is then concluded.
\end{proof}

We can now prove the analogous of Theorem \ref{lambda(m)} for linear groups.

\begin{thm}\label{lambda_F(m)} 
For every $m\geq 2$ and every field $\mathbb{F}$ with at least four elements, 
we have $\lambda_{\mathbb{F}} (m)=1+\lfloor \log_5(m/2)\rfloor$. 
When $\vert \mathbb{F}\vert \leq 3$, then 
$1+\lfloor \log_5(m/3)\rfloor \leq \lambda_{\mathbb{F}} (m)\leq 1+
\lfloor \log_5(m/2)\rfloor$. 
\end{thm}
\begin{proof} 
If $G\leq \mathrm{GL}(m,\mathbb{F})$ has $\lambda(G)=\lambda_\mathbb{F}(m)$, 
then $m\geq 2\cdot 5^{\lambda_\mathbb{F}(m) -1}$, by Lemma \ref{linear degree}. 
From this inequality we get $\lambda_\mathbb{F}(m)\leq 
1+\lfloor \log_5(m/2)\rfloor$.

Conversely, if $\mathbb{F}$ has order at least $4$ let 
$n=\lfloor \log_5(m/2)\rfloor$ and consider $W$ an $n$-fold wreath 
product of copies of $\Alt{5}$ in natural action, acting on a set $\Omega$ 
of cardinality $5^n$. The group $G=\mathrm{GL}(2,\mathbb{F})\wrr_\Omega{W}$ 
can be embedded into $\mathrm{GL}(m,\mathbb{F})$. 
Since $\mathbb{F}$ has at least four elements, $\lambda(\mathrm{GL}(2,\mathbb{F}))=1$
and therefore $\lambda(G)=n+1$. Hence 
$\lambda_\mathbb{F}(m)\geq 1+\lfloor \log_5(m/2)\rfloor$.\\
If $\left|\mathbb{F}\right|\leq 3$ we define 
$n=\lfloor \log_5(m/3)\rfloor$ and let $G$ be 
$\mathrm{GL}(3,\mathbb{F})\wrr_\Delta{X}$, where $X$ is the $n$-fold 
wreath product of alternating groups of degree $5$, acting naturally 
on the set $\Delta$ or order $5^n$. 
Clearly $\lambda(G)=1+\lfloor \log_5(m/3)\rfloor$. Hence in this case we have 
$$
 1+\lfloor \log_5(m/2)\rfloor\leq \lambda_\mathbb{F}(m)\leq 1+
   \lfloor \log_5(m/2)\rfloor\qquad \textrm{ when }\left|\mathbb{F}\right|\leq 3
$$
which completes the proof of the Theorem.
\end{proof}

We remark that $k=\lfloor \log_5(m/2)\rfloor>\lfloor \log_5(m/3)\rfloor$ 
if and only if  $2\cdot 5^k\leq m<3\cdot 5^k$.

\bigskip

We finish this section establishing a lower bound for the exponent 
of a group in $\Lambda_n$.
\begin{prop}\label{exponent} Let $G$ be a group in $\Lambda_n$. 
Then $G$ contains elements of order $2^n$ and therefore 
 $\exp(G)\geq 2^n$.
\end{prop}
\begin{proof}
Since the claim clearly holds for $n=0,1$, we argue by induction on $n$, 
assuming the result true for groups in $\Lambda_{n-1}$.
Let $G$ be any group in $\Lambda_n$. In order to prove the claim it is 
enough to find an element of order $2^n$ in a suitable section of $G$. 
Therefore we can assume that $G$ is $n$-rarefied and that $\Phi(G)=1$. 
Let $gR_2(G)$ be an element 
of order $2^{n-1}$ in $G/R_2(G)\in\Lambda_{n-1}$. 
If the $2$-part of $\vert g\vert$ is greater that $2^{n-1}$ 
there is nothing to prove, therefore we may assume without 
loss of generality that $g$ can be choosen of order exactely $2^{n-1}$, 
so that $\langle g\rangle \cap R_2(G)=1$.
Consider the action of $\langle g \rangle$ on the components of 
$S_1(G)=\prod_{i=1}^lS_i$. Note that there is of course at least 
one orbit of length $2^{n-1}$,
otherwise the nontrivial subgroup $\langle g^{2^{n-2}} \rangle$ 
will lie in $\bigcap_{i=1}^l N_G(S_i)\leq R_2(G)$, a contradiction.
Assume that $\Delta=\{S_i^{x}\vert x\in \langle g \rangle\}$ is an
orbit of length $2^{n-1}$. Then if $a$ is any element of $S_1(G)$ whose 
only nontrivial entry is in position $i$, we have
$$
(ag^{-1})^{m}=
aa^g a^{g^2}\cdots a^{g^{m-1}}g^{-m}.
$$
In particular for $m=2^{n-1}$ this is an element of the same order 
of $a$. Choosing $a$ of order $2$ we get the claim.
\end{proof}

\end{document}